\documentclass{amsart}

\usepackage{hyperref}
\usepackage{amsmath, amsthm, amssymb}
\usepackage{url}
\usepackage{graphicx}

\newtheorem{theorem}{Theorem}
\newtheorem{corollary}[theorem]{Corollary}
\newtheorem{lemma}[theorem]{Lemma}

\theoremstyle{definition}
\newtheorem*{remark}{Remark}

\newcommand{\abs}[2][]{#1\lvert #2 #1\rvert}

\title{When Are There Continuous Choices for the Mean Value Abscissa?}%

\author{David Lowry-Duda and Miles Wheeler}

\begin{document}

\begin{abstract}
  The mean value theorem of calculus states that, given a differentiable function $f$ on
  an interval $[a, b]$, there exists at least one mean value abscissa $c$ such that the
  slope of the tangent line at $c$ is equal to the slope of the secant line through $(a,
  f(a))$ and $(b, f(b))$.
  In this article, we study how the choices of $c$ relate to varying the right endpoint
  $b$.
  In particular, we ask: When we can write $c$ as a continuous function of $b$ in some
  interval?

  Drawing inspiration from graphed examples, we first investigate this question by proving
  and using a simplified implicit function theorem.
  To handle certain edge cases, we then build on this analysis to prove and use a
  simplified Morse's lemma.
  Finally, further developing the tools proved so far, we conclude that if $f$ is
  analytic, then it is always possible to choose mean value abscissae so that $c$ is a
  continuous function of $b$, at least locally.
\end{abstract}

\maketitle


\section{Introduction and Statement of the Problem.}\label{sec:intro}

The mean value theorem is one of the truly fundamental theorems of calculus. It says that
if $f$ is a differentiable function defined on a closed interval $[a,b]$, then there is at
least one $c$ in the open interval $(a,b)$ such that
\begin{equation}\label{eq:MVT}
  \frac{f(b) - f(a)}{b-a} = f'(c).
\end{equation}
We call $c$ a \emph{mean value abscissa} for $f$ on $[a, b]$.
Looking at a graph $y=f(x)$ as in Figure~\ref{fig:basic}, the left hand side
of~\eqref{eq:MVT} is the slope of the secant line from $(a,f(a))$ to
$(b,f(b))$, while the right hand side is the slope of the tangent line passing through
$(c,f(c))$. Observe that there can be multiple choices of $c$; in the first graph in
Figure~\ref{fig:basic} we could have chosen either $c$ or $c'$.
\begin{figure}[h]
  \includegraphics[scale=0.9]{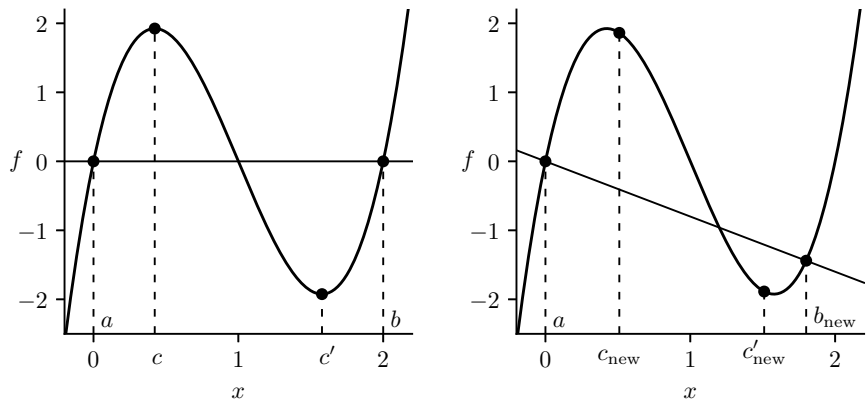}
  \centering
  \caption{%
  An illustration of the mean value theorem on the function $f(x) = x^3 - 3x^2 + 2x$.
  The straight lines are the secant lines.
  In each graph, the end-points of the secant line and two mean value abscissae are
  indicated by points on the curve.
  Dashed lines from each point indicate corresponding $x$-values.\label{fig:basic}}
\end{figure}

In this article we are interested in how the set of mean value abscissae $c$ changes as we
vary one of the endpoints of the interval, say the right endpoint $b$.
In particular, we are interested in the following problem: Suppose $c=c_0$ is our favorite
mean value abscissa for $a=a_0$ and $b=b_0$.
If $b$ changes slightly, can we also change $c$ slightly so that~\eqref{eq:MVT} is still
satisfied? In other words, \textbf{is there a locally continuous choice $c=C(b)$ of the
mean value abscissa?}
For example, in the right-hand graph in Figure~\ref{fig:basic}, we consider the new value
$b_{\mathrm{new}}$.
Here, it appears that the small change from $b$ to $b_{\mathrm{new}}$ corresponds to small
changes from $c$ to $c_{\mathrm{new}}$ and $c'$ to $c'_{\mathrm{new}}$ --- is this always
possible?

\section{Some examples.}\label{sec:examples}
To get a better feel for the problem we have set out for ourselves, let's
graph some functions and their mean value abscissae. To make our life simpler, we will
stick to examples with $a=a_0=0$ and $f(a_0)=f(b_0)=0$. Then the left hand side
of~\eqref{eq:MVT} is zero when $a=a_0$ and $b=b_0$, and so any corresponding mean value
abscissa $c=c_0$ has to be a critical point where $f'(c_0)=0$.
This may seem like a lot of assumptions, but in fact if someone hands us a more general
function $f$ we can always consider the related function
\begin{equation*}
  g(x) = f(x) - \Big( \frac{f(b_0) - f(a_0)}{b_0 - a_0}(x - a_0) + f(a_0) \Big),
\end{equation*}
which satisfies $g(a_0) = g(b_0) = 0$. Both $f$ and $g$ have the same solutions to the
mean value condition~\eqref{eq:MVT}.

\begin{figure}[h]
  \centering
  \includegraphics[scale=0.9]{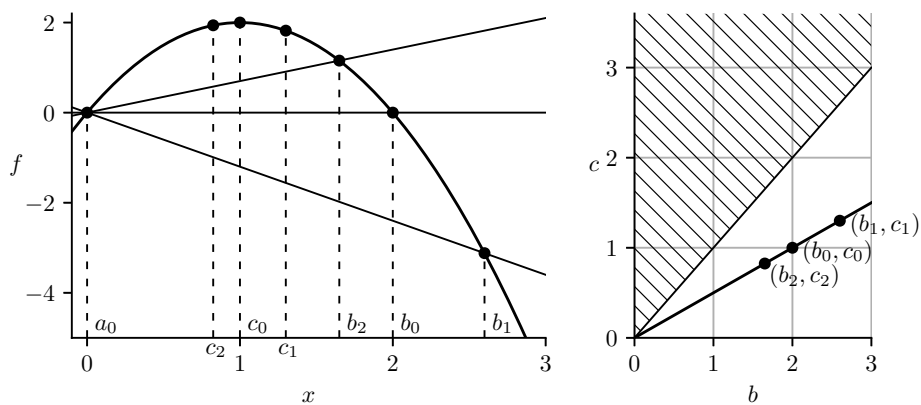}
  \caption{
    \textbf{Left}: the parabola $y = -x^2 + 2x$.
    Three pairs of points are shown: $(b_0, c_0)$, $(b_1, c_1)$, and $(b_2, c_2)$.
    For each $b_i$, the corresponding $c_i$ is a mean value abscissa on the interval
    $[a_0, b_i]$.
    \textbf{Right}: the graph of \emph{all} mean value abscissae as a function of
    $b$ (where $a_0 = 0$ is fixed); $b$ is on the horizontal axis, and $c = b/2$ is on the
    vertical axis.
    Points corresponding to the three pairs on the left are noted.
    In the mean value theorem, $b > c$, and we represent this by shading the region where
    $c \geq b$.
  \label{fig:parabola}
  }
\end{figure}
Consider the parabola at the left of Figure~\ref{fig:parabola}.
There is only one choice of $c_0$: the vertex.
When we slightly increase $b_0$ to $b_1$, we have to slightly increase $c_0$ to $c_1$.
Similarly if we decrease $b_0$ to $b_2$, then we have to decrease $c_0$ to $c_2$.
Plotting the mean value abscissae $c$ for each $b$, we get the picture at the
right of Figure~\ref{fig:parabola}.
Looking at the figure, $c$ seems to be a continuous function of $b$, and indeed in this
case we can solve~\eqref{eq:MVT} explicitly to get $c=b/2$.
In particular, the ratio $c/b$ is constant; for more on the class of functions with this
property, see~\cite{carter2017}.

\begin{figure}[h]
  \centering
  \includegraphics[scale=0.9]{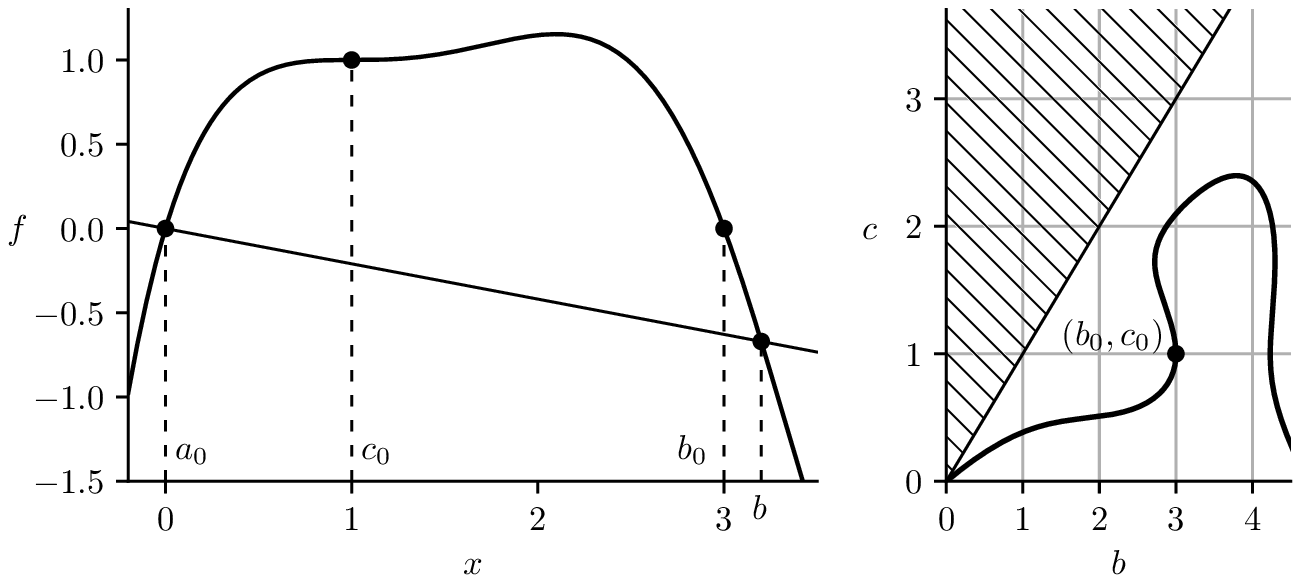}
  \caption{
    \textbf{Left}: the graph of a function with an inflection point.
    The point $(1, 1)$ is a mean value abscissa on the interval $[0, 3]$, but there is no
    continuous extension of this solution to a neighborhood of $b_0$.
    The straight line is the secant line corresponding to the interval $[a_0, b_1]$.
    \textbf{Right}: the graph of all mean value abscissae as a function of $b$, as in the
    previous Figure.
    The behavior is substantially more complicated.
    At the point $(b_0, c_0)$, we observe that $c$ is not a function of $b$.
  \label{fig:curlym}
  }
\end{figure}
Next consider the more complicated graph at left in Figure~\ref{fig:curlym}.
There are now two critical points.
One is a local maximum, and the behavior near this point is very similar to the behavior
near the vertex of the parabola.
The second one, which we have labeled as $c_0$, is a non-extremal critical point (it is
neither a local maximum nor a local minimum).
Suppose that $b_1$ is just a bit bigger than $b_0$.
Then the slope of the secant line which appears on the right hand side of~\eqref{eq:MVT}
is $<0$.
When $c$ is close to $c_0$, on the other hand, the right hand side $f'(c)$
of~\eqref{eq:MVT} is $\geq 0$.
There is no solution to~\eqref{eq:MVT} without choosing $c$ far away from $c_0$.

\section{The implicit function theorem.}

\subsection{Implicit equations.}
In the last section, we saw that the set of solutions $(b,c)$ of~\eqref{eq:MVT} can
look quite complicated. On the right of Figure~\ref{fig:curlym}, for instance, $c$ is not
a function of $b$ (the curve fails the ``vertical line test'') and $b$ is also not a
function of $c$ (the curve fails the ``horizontal line test''). 
This is possible because~\eqref{eq:MVT} is an \emph{implicit} equation.

Implicit equations show up all over the place in mathematics, for instance in geometry;
$x^2+y^2=1$ is the equation for the circle with radius one centered at the origin, while
$x^2 + 4xy + y^2 = 2$ is the equation for a certain hyperbola. By subtracting off the right
hand side, we can write any implicit equation in two variables as
\begin{align}
  \label{eqn:implicit}
  F(x,y) = 0
\end{align}
for some function $F$. It's often tempting to try and solve an implicit equation for
one of the variables, and indeed that's indeed how we got the formula
$c=b/2$ for the example on the right of Figure~\ref{fig:parabola}. When the equation gets
more complex, though, like it is in Figure~\ref{fig:curlym}, this method
becomes very tedious and is quite often impossible!

\subsection{The implicit function theorem.}

Thankfully, calculus offers us a powerful tool, called the \emph{implicit function
theorem}, to help us understand implicit equations.
The intuition behind the theorem is the following: Suppose that we have found one solution
$(x_0,y_0)$ to~\eqref{eqn:implicit}, and are only interested in solutions
of~\eqref{eqn:implicit} nearby this initial solution.
If the function $F$ is differentiable, then for $(x,y) \approx (x_0,y_0)$ we can
approximate $F$ as
\begin{align}
  \label{eqn:taylor}
  F(x,y) \approx F(x_0,y_0) + F_x(x_0,y_0)(x-x_0)
  + F_y(x_0,y_0)(y-y_0),
\end{align}
where the subscripts on $F$ are partial derivatives.
This is a first-order Taylor approximation of $F$, a two-variable version of the
tangent-line approximation for functions of a single variable.
Plugging~\eqref{eqn:taylor} into the equation~\eqref{eqn:implicit} that we are trying to
solve and using $F(x_0,y_0) = 0$, we get
\begin{align}
  \label{eqn:tayloreqn}
  F_x(x_0,y_0)(x-x_0)
  + F_y(x_0,y_0)(y-y_0) \approx 0.
\end{align}
While~\eqref{eqn:tayloreqn} is only approximately true, it's advantage is that it's a
\emph{linear} equation.
In particular, if $F_y(x_0,y_0) \ne 0$, then we can try to ``solve for $y$'', giving the
approximation
\begin{align}
  \label{eqn:taylorsolved}
  y = Y(x) \approx y_0 - \frac{F_x(x_0,y_0)}{F_y(x_0,y_0)}(x-x_0).
\end{align}
The implicit function theorem says that the conclusion of this intuitive argument is
nearly correct: as long as $F(x_0,y_0) = 0$ and $F_y(x_0,y_0) \ne 0$, we can
indeed solve $F(x,y)=0$ for $y$ when $(x,y)$ is close to $(x_0,y_0)$.

\begin{theorem}[Implicit Function Theorem]\label{thm:implicit}
  Suppose that $F=F(x,y)$ is a continuously differentiable function and that at some point
  $(x_0,y_0)$ we have
  \begin{align}
    \label{eqn:firstorder}
    F(x_0,y_0) = 0
    \quad\text{and}\quad
    F_y(x_0,y_0) \neq 0.
  \end{align}
  Then there exist $\varepsilon > 0$, $\delta > 0$, and a continuously differentiable
  function $Y(x)$ such that the implicit equation $F(x,y) = 0$ is equivalent to the
  explicit equation $y=Y(x)$ whenever $\abs{x-x_0} < \delta$ and $\abs{y-y_0} <
  \varepsilon$.
\end{theorem}

\subsection{Proof of the implicit function theorem.}

The implicit function theorem above is an \emph{existence} theorem: it says there
\emph{exists} a function $Y=Y(x)$ with some special properties. Like many of the existence
theorems in calculus,  the implicit function theorem has a nice proof using the
contraction mapping principle. The implicit function theorem presented here is a
simplified version, but the proofs of more general versions share have the same basic
outline; see, for instance, \cite[\S13]{strichartz}.

The contraction mapping principle, which is also called Banach's fixed-point theorem,
concerns equations of the form
\begin{align}
  \label{eqn:fixedpoint}
  y = K(y), \qquad \text{$y$ in $I$}
\end{align}
where $I=[A,B]$ is a closed interval.
We call~\eqref{eqn:fixedpoint} a ``fixed-point equation'' because it says that the point
$y$ is ``fixed'' (or unchanged) when we apply the function $K$ to it. The theorem assumes
that the function $K$ satisfies
\begin{align}
  \label{eq:contraction}
  \abs{ K(y) - K(y') } &\le \rho \abs{ y - y' }
  \qquad \text{for all $y,y'$ in $I$}
\end{align}
for some constant $\rho < 1$. If $y,y'$ are two points a distance $d$ apart,
then~\eqref{eq:contraction} says that the distance between their images $K(y),K(y')$ is at
most $\rho d$. Since $\rho d < d$, the points are closer together after we apply $K$, and
so we
call $K$ a \emph{contraction}.

\begin{theorem}[Contraction Mapping Principle]\label{thm:contraction}
  Suppose that the function $K$ is defined on a closed interval $I$ where it
  satisfies~\eqref{eq:contraction} for some constant $\rho < 1$. If $K(y)$ lies in $I$ for
  every $y$
  in $I$, then the fixed-point equation~\eqref{eqn:fixedpoint} has a unique solution
  $y^*$.
\end{theorem}
\begin{proof}
  First we show that solutions of~\eqref{eqn:fixedpoint} are unique.
  Suppose that $y$ and $y'$ both solve~\eqref{eqn:fixedpoint}, so that they
  satisfy $y = K(y)$ and $y = K(y')$. Then by~\eqref{eq:contraction} we see that
  \begin{align*}
    \lvert y - y' \rvert = \lvert K(y) - K(y')
    \rvert \le \rho \lvert y - y' \rvert.
  \end{align*}
  Since $\rho < 1$, this is only possible if $y=y'$.

  Having shown uniqueness of a potential solution, we now show that~\eqref{eqn:fixedpoint}
  has a solution $y^*$. Choose any point $y_0$ in $I$ and define the sequence
  $y_1,y_2,y_3,\ldots$
  recursively by
  \begin{align}
    \label{eqn:recursion}
    y_{n+1} = K(y_n).
  \end{align}
  Since $K$ sends points in $I$ to points in $I$, this definition makes sense and we can
  prove by induction that $y_n$ lies in $I$ for all $n$. We will show that $\lim_{n\to
  \infty} y_n$ exists, and that it is the fixed point $y^*$ we are looking for.

  By repeatedly using~\eqref{eq:contraction} and~\eqref{eqn:recursion}, we can estimate
  the distance between successive terms $y_{n+1}$ and $y_n$ in our sequence:
  \begin{equation}\label{eq:contraction_ineq}
    \begin{split}
      \lvert y_{n+1} - y_n \rvert
      &= \lvert K(y_{n}) - K(y_{n-1}) \rvert \\
      &\le \rho \lvert y_n - y_{n-1} \rvert \\
      & = \rho \lvert K(y_{n-1}) - K(y_{n-2}) \rvert \\
      &\le \rho^2 \lvert y_{n-1} - y_{n-2} \rvert \\
      & \le \cdots \le \rho^{n-1} \lvert y_1 - y_0 \rvert.
    \end{split}
  \end{equation}
  Since $\rho < 1$, the right hand side converges to zero very quickly as $n \to \infty$.
  If $n > m$, we can then repeatedly use~\eqref{eq:contraction_ineq} to estimate the difference
  between $y_n$ and $y_m$:
  \begin{align*}
    \lvert y_n - y_m \rvert
    &=  \lvert (y_n - y_{n-1}) + (y_{n-1} - y_{n-2}) + \cdots + (y_{m+1} - y_m) \rvert\\
    &\le  \abs{y_n - y_{n-1}} + \abs{y_{n-1} - y_{n-2}} + \cdots + \abs{y_{m+1} - y_m} \\
    & \le \lvert y_1 - y_0 \rvert (\rho^{n-1} + \rho^{n-2} + \cdots + \rho^m)\\
    & = \lvert y_1 - y_0 \rvert \rho^m \frac{1-\rho^n}{1-\rho}\\
    & < \lvert y_1 - y_0 \rvert \frac{\rho^m}{1 - \rho}.
  \end{align*}
  Here in the second step we have used the triangle inequality, and in the second-to-last step we
  have used the formula for the (partial) sum of a geometric series. As before the right
  hand side converges to 0 as $m \to 0$, which now shows that the sequence $\{y_n\}$ is
  Cauchy. In particular, the limit $y^* = \lim_{n\to\infty} y_n$ exists. Since each
  $y_n$ lies in the closed interval $I$, the same is true for the limit $y^*$.

  Finally, we note that~\eqref{eq:contraction} implies that $K$ is continuous. Taking the
  limit of the recurrence~\eqref{eqn:recursion} as $n \to \infty$, we therefore get $y^* =
  K(y^*)$, i.e.~that $y^*$ solves~\eqref{eqn:fixedpoint}.
\end{proof}

To use Theorem~\ref{thm:contraction} to prove Theorem~\ref{thm:implicit}, we first
rewrite $F(x,y)=0$ as a fixed-point equation $y=K(y;x)$ for $y$. When $x$ is close to
$x_0$ and $I$ is a small interval centered at $y_0$, we will then show that this $K$ is a
contraction mapping satisfying the hypotheses of Theorem~\ref{thm:contraction}. Here the
notation $K(y;x)$ is to remind us that $y$ is the main variable while $x$ is a parameter.

To keep things simple, we only prove the implicit function theorem for the special
case $x_0=y_0=0$. It's easy to prove the general case from this specific one by
considering the shifted function $G(x,y) = F(x_0+x,y_0+y)$. The inspiration for the proof
is our informal argument which lead to~\eqref{eqn:taylorsolved}. This argument started
with the Taylor expansion~\eqref{eqn:taylor}, but to make it rigorous we start
with an exact version of that formula:
\begin{align}
  \label{eqn:taylorexact}
  F(x,y) = F_x(0,0)x
  + F_y(0,0)y +
  r(x,y).
\end{align}
Here $r(x,y)$ is the remainder term, which is small when $(x,y)$ is close to $(0,0)$, and
we have used that $x_0=y_0=0$ and $F(x_0,y_0) = F(0,0) = 0$.
To devise a mapping $K$, we set $F(x,y)=0$ and do some algebra to bring the $y$ to the
left hand side to get
\begin{align}
  \label{eqn:prefixed}
  y = -F_y(0,0)^{-1} \Big(F_x(0,0)x + r(x,y)\Big).
\end{align}
Solving~\eqref{eqn:taylorexact} for $r(x,y)$ and plugging into~\eqref{eqn:prefixed},
things simplify a bit and we get
\begin{align}
  \label{eqn:postfixed}
  y = y-F_y(0,0)^{-1} F(x,y).
\end{align}
We see that~\eqref{eqn:postfixed} is true if and only if $F(x,y) = 0$. What we
have gained, though, is that~\eqref{eqn:postfixed} gives a fixed-point equation for $y$,
and so we can hope to solve it for $y$ by applying Theorem~\ref{thm:contraction} with
\begin{align}
  \label{eqn:ourcontraction}
  K=K(y;x) =  y-F_y(0,0)^{-1} F(x,y).
\end{align}

We apply Theorem~\ref{thm:contraction} with $\rho = 1/2$ and $I =
[-\varepsilon,\varepsilon]$ for some small number $\varepsilon > 0$ which we still have to
determine. This ensures that we are only considering $y$-values which are close to
$y_0=0$. We will also restrict ourselves to $x$-values which are close to $x_0=0$, say $x$
in $[-\delta,\delta]$ for some other small number $\delta > 0$. The hypotheses of
Theorem~\ref{thm:contraction} will therefore be met as long as
\begin{alignat}{2}
  \label{eqn:kcontain}
  \abs{K(y;x)} &\le \varepsilon  &\qquad& \text{for $\abs x \le \delta$, $\abs y \le \varepsilon$},\\
  \label{eqn:kcontract}
  \abs{K(y;x)-K(y';x)} &\le \tfrac 12 \abs{y-y'} && \text{for $\abs x \le \delta$, $\abs y, \abs{y'} \le \varepsilon$}.
\end{alignat}
The second inequality~\eqref{eqn:kcontract} is the contraction
condition~\eqref{eq:contraction}, while the first~\eqref{eqn:kcontain} guarantees that
$K(y;x)$ lies in $I$ whenever $y$ does.

We'll prove~\eqref{eqn:kcontract} using, of all things, the mean value theorem.
Differentiating with respect to $y$ we get
\begin{equation*}
  K_y(y;x) = 1 - F_y(0,0)^{-1} F_y(x,y).
\end{equation*}
It's clear that at $(0,0)$, the right hand side is $0$.
Since $F_y$ is continuous, we can pick $\varepsilon > 0$ small enough
that the right hand side is bounded by $1/2$ whenever $\abs x, \abs y \le \varepsilon$.
Now suppose that $\abs x, \abs y, \abs {y'} \le \varepsilon$. Applying the mean
value theorem to $K_x$ on the interval between $y$ and $y'$, we get that
\begin{equation}\label{eq:K_contraction}
  \lvert K(y';x) - K(y;x) \rvert
  =
  \lvert K_y(c;x) \rvert \lvert y' - y \rvert
  \leq \tfrac 12 \lvert y' - y \rvert
\end{equation}
for some point $c$ between $y$ and $y'$. This shows~\eqref{eqn:kcontract}.

We still need to show~\eqref{eqn:kcontain}. As long as $\abs x, \abs y \le
\varepsilon$, we can use~\eqref{eqn:kcontract} to estimate
\begin{align}
  \label{eqn:containshow}
  \begin{aligned}
    \lvert K(y;x) \rvert
    &\le \lvert K(y;x) - K(0;x) \rvert + \lvert K(0;x) \rvert \\
    &\le \tfrac 12 \abs y + \abs{K(0;x)}\\
    &\le \tfrac 12 \varepsilon + \abs{K(0;x)}.
  \end{aligned}
\end{align}
Since $K(0;x)$ is a continuous function of $x$ and $K(0;0) = 0$, there exists a $\delta
> 0$ so that $\abs{K(0;x)} \le \varepsilon/2$ whenever $\abs x \le \delta$. Picking
$\delta$ smaller if necessary so that $\delta \le \varepsilon$,~\eqref{eqn:containshow}
finally implies that $\abs{K(y;x)} \le \frac 12 \varepsilon + \frac 12\varepsilon =
\varepsilon$ whenever $\abs x \le \delta$ and $\abs y \le \varepsilon$.

Now that we have finished proving~\eqref{eqn:kcontain} and~\eqref{eqn:kcontract}, we
can apply Theorem~\ref{thm:contraction} to guarantee that the fixed point equation
$y=K(y;x)$ has a unique solution $y=Y(x)$ in $[-\varepsilon,\varepsilon]$ for each $x$ in
$[-\delta,\delta]$. Since the fixed-point equation $y=K(y;x)$ is equivalent to $F(x,y) =
0$, this completes the proof of the theorem except for the claim that $Y$ is a
continuously differentiable function.

To see that $Y(x)$ is continuous, we write
\begin{equation*}
  \begin{split}
    \lvert Y(x') - Y(x) \rvert &= \lvert K(Y(x');x') - K_x (Y(x);x) \rvert
    \\
    &\leq
    \lvert K(Y(x');x') - K(Y(x);x') \rvert + \lvert K(Y(x);x') - K(Y(x);x) \rvert.
  \end{split}
\end{equation*}
By~\eqref{eq:K_contraction}, the first term on the right hand side is bounded
by $\frac{1}{2}\lvert Y(x') - Y(x) \rvert$.
Rearranging, this shows that
\begin{equation*}
  \lvert Y(x') - Y(x) \rvert \leq 2 \lvert K(Y(x);x') - K(Y(x);x) \rvert.
\end{equation*}
The continuity of $Y$ then follows from the continuity of $K(y;x)$ as a function of $x$.

Next we show that $Y$ is continuously differentiable and calculate its derivative. If we
knew ahead of time that $Y$ was differentiable, then we could solve for $Y'(x)$ by
differentiating $F(x,Y(x))=0$ using the chain rule. Since we do not know yet that $Y$ is
differentiable, we instead look at the difference quotient
\begin{align*}
  0 = \frac{F(x+h,Y(x+h))-F(x,Y(x))}h.
\end{align*}
Using the fundamental theorem of calculus in a clever way, we rewrite the numerator of
this difference quotient as
\begin{align*}
  0 &= F(x+h, Y(x+h)) - F(x, Y(x)) \\
  &= \int_0^1 \frac{d}{dt} F(x + th, tY(x+h) + (1-t)Y(x)) dt \\
  &= h \int_0^1 F_xdt + (Y(x+h) - Y(x)) \int_0^1 F_y  dt,
\end{align*}
where the arguments of both $F_x$ and $F_y$ are $(x + ty, tY(x+h) + (1-t) Y(x))$.
Notice that the chain rule has caused a $(Y(x+h) - Y(x))$ to appear.
We rearrange this into an expression for the difference quotient
\begin{equation*}
  \frac{Y(x+h) - Y(x)}{h} = - \bigg[ \int_0^1 F_y dt \bigg]^{-1} \int_0^1 F_x dt.
\end{equation*}
Since $F_y(0,0) \neq 0$, for $h$ small enough the integral of
$F_y$ will not vanish, and so it is valid to divide by it.

Now we take the limit as $h \to 0$. On the left hand side, this gives $Y'(x)$ directly. On
the right hand side, we have to pass the limit inside the integrals. Since $(x+th, tY(x+h)
+ (1-t)Y(x))$ converges uniformly to $(x, Y(x))$ as $h \to 0$, this is justified and we
get
\begin{equation}\label{eq:ift_derivative}
  Y'(x) = - \frac{F_x(x,Y(x))}{F_y(x,Y(x))}.
\end{equation}
This proves that $Y$ is differentiable. Since $Y$ is continuous and $F$ is continuously
differentiable, the right hand side of~\eqref{eq:ift_derivative} is continuous, and so $Y$
is in fact continuously differentiable. Looking at~\eqref{eq:ift_derivative}, we also see
that this justifies the approximate formula~\eqref{eqn:taylorsolved} as we had hoped!

By repeatedly differentiating~\eqref{eq:ift_derivative}, we discover that if $F$ is
$k$-times continuously differentiable, then so is $Y$. We record this observation as a
corollary to the proof of Theorem~\ref{thm:implicit}.

\begin{corollary}\label{cor:higherdiff}
  In Theorem~\ref{thm:implicit}, the derivative $Y'$ is given by~\eqref{eq:ift_derivative}. 
  If $F$ is $k$-times continuously differentiable, then $Y$ is
  $k$-times continuously differentiable.
\end{corollary}

\subsection{Application to the Mean Value Abscissa.}

With the implicit function theorem in hand, we are now ready to investigate the
possibility of determining when there exist locally continuous choices of the mean value
abscissa $c$ in~\eqref{eq:MVT}. The first step is to rewrite~\eqref{eq:MVT} as $F(b,c) =
0$ where
\begin{equation}
  \label{eq:F_def}
  F(b,c) =  \frac{f(b) - f(a)}{b-a} - f'(c).
\end{equation}
From now on we assume that $f$ is twice continuously differentiable, in which case $F$ is
once continuously differentiable.

Suppose that $c_0$ is a mean value abscissa corresponding to $b_0$, i.e.~that
$F(b_0,c_0) = 0$. A quick computation shows that
\begin{equation*}
  F_b(b_0,c_0)
  = \frac{f'(b_0) - f'(c_0)}{b_0-a}, \qquad F_c(b_0,c_0) = - f''(c_0).
\end{equation*}
Thus $F_c(b_0,c_0) \neq 0$ is true exactly when $f''(c_0) \neq 0$.
And if $f''(c_0) \neq 0$, then by Theorem~\ref{thm:implicit} there exists $\varepsilon > 0$,
$\delta > 0$, and a continuously differentiable function $C(b)$ so that $F(b,c) = 0$ is
equivalent to $c=C(b)$ whenever $\abs{c-c_0} < \varepsilon$ and $\abs{b-b_0} < \delta$.

Although we have focused on the question of when the mean value abscissa $c$ can be
written as a continuous function of the right endpoint $b$, we also have the data for the
converse question: When is the right endpoint $b$ a function of the mean value abscissa
$c$? By Theorem~\ref{thm:implicit}, $b$ can be written as a function of $c$ near
$(b_0,c_0)$ when $F_b(b_0,c_0) \neq 0$, or equivalently when $f'(b_0) \neq f'(c_0)$.

In total we have proved the following theorem.
\begin{theorem}\label{thm:mainI}
  Let $f$ be a twice continuously differentiable function, fix an interval
  $[a_0,b_0]$, and let $c_0$ be a mean value abscissa for $f$ on $[a_0,b_0]$.
  \begin{enumerate}
  \item[(a)] Suppose that $f''(c_0) \ne 0$. Then there is a continuously differentiable
    function $C(b)$ so that
    \begin{align*}
      \frac{f(b)-f(a)}{b-a} = f'(C(b))
    \end{align*}
    for all $b$ close to $b_0$. There are no other solutions $(b,c)$ of~\eqref{eq:MVT}
    close to $(b_0,c_0)$.
  \item[(b)] Suppose that $f'(b_0) \neq f'(c_0)$. Then there is a continuously
    differentiable function $B(c)$ so that
    \begin{align*}
      \frac{f(B(c))-f(a)}{B(c)-a} = f'(c)
    \end{align*}
    for all $c$ close to $c_0$. There are no other solutions $(b,c)$ of~\eqref{eq:MVT}
    close to $(b_0,c_0)$.
  \end{enumerate}

\end{theorem}

\begin{remark}
  Given an initial solution $(b_0,c_0)$, our proof of Theorem~\ref{thm:mainI} is
  \emph{constructive} in that iterating the contraction map from the proof
  Theorem~\ref{thm:implicit} actually gives us an algorithm for approximating $B(c)$ or
  $C(b)$ to any order of accuracy.
\end{remark}

This theorem gives perspective on Figure~\ref{fig:curlym}. The mean value abscissa in that
figure was at an inflection point, where $f''(c_0) = 0$, and so Theorem~\ref{thm:mainI}a is
inconclusive about whether we can write $c=C(b)$. Looking at the figure it appears that we
cannot. On the other hand $f'(b_0) < f'(c_0) = 0$, and so Theorem~\ref{thm:mainI}
implies that we \emph{can} write $b=B(c)$.

\section{The Morse Lemma.}\label{sec:morse}

We have now shown that there exist continuous choices of $c=C(b)$ around those mean value
abscissae such that $f''(c_0) \neq 0$. Conversely, we've shown that when there is a mean
value abscissa $c$ such that $f'(b_0) \neq f'(c_0)$, then $b$ can be written as a
continuous function in a neighborhood of $c_0$.

But what if both $f''(c_0) = 0$ and $f'(b_0) = f'(c_0)$? As before, we return to pictorial
investigation. Fortunately, these are two strong constraints and we quickly identify
interesting aspects from graphs.

\begin{figure}[h]
  \includegraphics[scale=0.9]{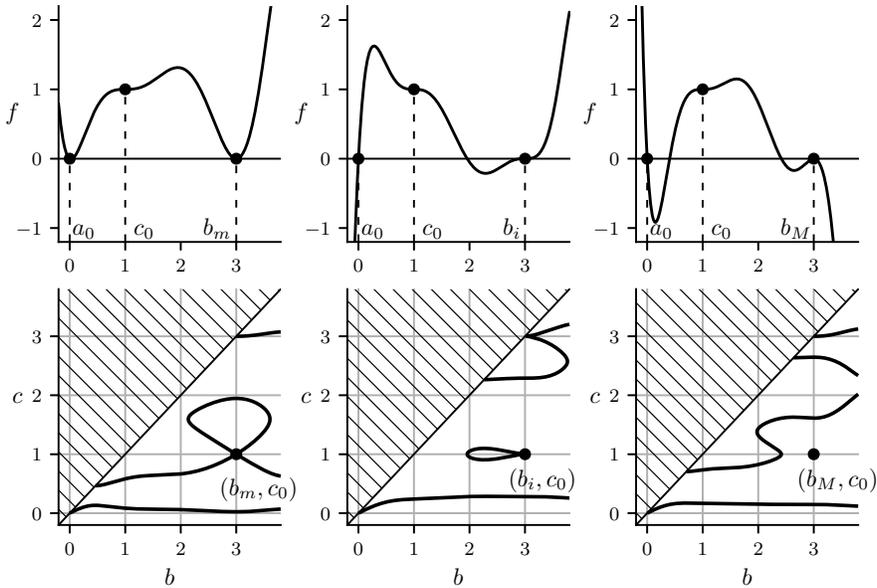}
  \centering
  \caption{%
  \textbf{Top}: Three graphs of functions $f$ with an interval and corresponding
  mean value abscissa indicated.
  In each graph, $f''(c_0) = 0$.
  In the left graph, $f''(b_m) > 0$.
  In the middle graph, $f''(b_i) = 0$.
  In the right graph, $f''(b_M) < 0$.
  \textbf{Bottom}: Below each graph is a plot of all mean value abscissae as a function of
  $b$, as in previous figures.
  In the first graph, there appear to be multiple choices of continuous function $c(b)$.
  In the second graph, there is a continuous extension on an interval with $b_i$ as an endpoint.
  In the third graph, the initial solution is completely isolated.
  \label{fig:triple}}
\end{figure}

For ease, we suppose again that $f(a_0) = f(b_0) = 0$, and we now suppose that $f'(c_0) =
f'(b_0) = f''(c_0) = 0$.
In Figure~\ref{fig:triple}, we examine three different functions $f$: each
satisfies$f''(c_0) = 0$, but $f''(b)$ is positive on the left, zero in the middle, and
negative on the right.
We've named these three values of $b$ as $b_m$, $b_i$, and $b_M$ (according to whether $b$
is a minimum, an inflection point, or a maximum, respectively).

Examining the top left graph of Figure~\ref{fig:triple}, we observe that in a small
neighborhood of $c_0$, all tangent lines have nonnegative slope.
Similarly, for all $b$ in a small neighborhood around $b_m$, the secant lines from $(0,0)$
to $(b, f(b))$ have nonnegative slope.
Qualitatively, it appears that for $b'$ just a little less than $b_m$, we could vary $c$
to match slopes.
\emph{But which direction should $c$ be moved?}
We can see this apparent choice of direction in the mean value abscissa graph at bottom
left: near $(b_m, c_0)$, the graph resembles an X.

This reveals a key difference to the situation when $f''(c) \neq 0$.
In both the implicit function theorem and Theorem~\ref{thm:mainI}, the resulting
implicitly defined functions are unique.
This is due to the uniqueness of the fixed points in the contraction mapping principle.
But here, it appears that sometimes there are multiple different continuous choices of
$c(b)$ --- that is, if there are any at all.

In the top right graph of Figure~\ref{fig:triple}, we see that in a small neighborhood of
$c_0$, all tangent lines again have nonnegative slope.
But in a small neighborhood around $b_M$, the secant lines from $(0,0)$ to $(b, f(b))$ all
have nonpositive slope.
Thus there is no hope to extend $c$ to a function to a larger neighborhood at all.
We recognize this in the mean value abscissa graph below by seeing that $(b_M, c_0)$ is an
isolated point.

The behavior in the top center graph, near $b_i$, is a bit more delicate.
Here, in a small neighborhood of $c_0$, all tangent lines have nonpositive slope.
For $b$ just to the left of $b_i$, the secant lines from $(0,0)$ to $(b,f(b))$ have
nonpositive slope, and so it qualitatively appears that it might be possible to associate
points near $c$ with matching slopes.
But for $b$ just to the right of $b_i$, the secant lines all have positive slope, which
cannot be matched to slopes of points in a neighborhood of $c$.

These examples indicate a wider variety of behavior, and it's not at all obvious what the
general rule should be.
We cannot hope to directly apply an implicit function theorem without some significant
changes.

As with our investigation of the implicit function theorem, let us start with the Taylor
expansion of $F(b,c)$ at $(b_0,c_0)$. As $F(b_0,c_0) = F_b(b_0,c_0) = F_c(b_0,c_0) = 0$,
all the terms in this expansion are at least quadratic. 
For simplicity, let's assume that two of the quadratic terms are nonzero, more
specifically that the partial derivatives $F_{bb}(b_0, c_0) \neq 0$ and $F_{cc}(b_0, c_0)
\neq 0$.
In this case, there is a result called the Morse lemma which is perfectly tailored for our
situation!
A simple version of the Morse lemma is the following.
\begin{lemma}[Morse lemma]\label{lem:morse}
  Let $G=G(x,y)$ be a three-times continuously differentiable function and suppose that
  $G(0,0)=G_x(0,0)=G_y(0,0)=0$ but that
  \begin{align}
    \label{eq:hessian}
    G_{xx}(0,0) G_{yy}(0,0) - (G_{xy}(0,0))^2 \ne 0.
  \end{align}
  Then in a neighborhood of the origin there is a change of coordinates $(x,y) \mapsto (u,v)$
  so that
  \begin{align}
    \label{eq:morse}
    G(x,y) = \pm u^2 \pm v^2.
  \end{align}
\end{lemma}
The number of minus signs on the right hand side of \eqref{eq:morse} is called the
\emph{Morse index} of $G$ at $0$. It is independent of the particular choice of
coordinates $(u,v)$, and is one of the basic ingredients in Morse theory~\cite{matsumoto}.
By a ``change of coordinates'' $(x, y) \mapsto (u, v)$, we mean that $u$ and $v$ can be
written as continuously differentiable functions of $(x, y)$, while at the same time $x$
and $y$ can be written as continuously differentiable functions of $(u, v)$. We also
require that $(0,0) \mapsto (0,0)$.

\begin{remark}
  Those familiar with multivariable calculus might recognize the conditions of the
  Morse lemma as an alternate way of saying that the gradient of $G$ vanishes at the
  origin, but the Hessian matrix is invertible there.
  A full proof of the Morse lemma involves the implicit function theorem in higher
  dimensions (or its close cousin the inverse function theorem).
  But we will see below that in our special case, Theorem~\ref{thm:implicit} is
  sufficient.
\end{remark}

We consider the function $G(x, y) = F(b_0 + x, c_0 + y)$, which effectively translates our
focus to the origin. Just as a solution to $F(b, c) = 0$ corresponds to a mean value
abscissa, a solution to $G(x, y) = 0$ also corresponds to a mean value abscissa; in
particular, $G(0, 0) = 0$ corresponds to the given mean value abscissa $c_0$ on the
interval $[a_0, b_0]$. Our assumptions on the partial derivatives of $F$ at $(b_0,c_0)$
similarly translate $G_x(0,0) = G_y(0,0) = 0$ while $G_{xx}(0,0) \ne 0$ and $G_{yy}(0,0)
\ne 0$. A quick calculation shows that $G_{xy}(0,0) = 0$ --- indeed $G_{xy}(x,y)$ is
always zero! Thus \eqref{eq:hessian} is satisfied, and, if $f$ is four-time continuously
differentiable so that $G$ is three-times continuously differentiable, we can apply the
Morse lemma.

In fact, our situation is a bit simpler than the one covered by the Morse lemma, and so we
will only prove the special case of the lemma that we need.
What's special is that $G$ naturally splits into a function depending only on $x$ and a
function depending only on $y$:
We can write $G(x, y) = g_1(x) - g_2(y)$, where
\begin{equation}\label{eq:g1g2}
  g_1(x) = \frac{f(b_0+x) - f(a_0)}{(b_0+x) - a_0} - f'(c_0),
  \qquad
  g_2(y) = f'(c_0+y) - f'(c_0).
\end{equation}
Thus $G(x, y) = 0$ is equivalent to $g_1(x) = g_2(y)$.
We include the $f'(c_0)$ terms in \eqref{eq:g1g2} so that $g_1(0) = g_2(0) = 0$, so we can continue
to focus our attention on the origin. Our assumptions $G_x(0,0) = G_y(0,0)$ now translate
into $g_1'(0) = g_2'(0)$, while our assumptions $G_{xx}(0,0) \ne 0$ and $G_{xy}(0,0) \ne
0$ translate into $g_1'(0) \ne 0$ and $g_2'(0) \ne 0$.

Taylor expanding $g_1$ and $g_2$ gives the approximations
\begin{equation*}
  g_1(x) \approx \frac{g_1''(0)}{2!} x^2 = \frac{1}{b_0 - a_0} \frac{f''(b_0)}{2!} x^2,
  \qquad
  g_2(y) \approx \frac{g_2''(0)}{2!} y^2 = \frac{f'''(c_0)}{2!} y^2,
\end{equation*}
and hence the approximation $G(x, y) \approx \alpha x^2 + \beta y^2$
with $\alpha = g_1''(0) / 2$ and $\beta = g_2''(0) / 2$.
We will show that we can choose coordinates $u$ and $v$ to make this approximation
\emph{exact} while at the same time taking the constants to be $\pm 1$.

To do this, we use Taylor's theorem to write $g_1$ and $g_2$ exactly as
\begin{align*}
  g_1(x) &= \alpha x^2 + r_1(x) x^2 = x^2\big(\alpha + r_1(x)\big) \\
  g_2(y) &= \beta y^2 + r_2(y) y^2 = y^2\big(\beta + r_2(y)\big),
\end{align*}
where $r_1(x)$ and $r_2(y)$ are remainder terms. Thus $r_1(x)$ is small when $x$ is
near $0$ and $r_2(y)$ is small when $y$ is near $0$.
Ideally, we would like to take coordinates like $u = x\sqrt{\alpha + r_1(x)}$ and $v = y
\sqrt{\beta + r_2(y)}$, so that $G(x, y) = u^2 - v^2$ (whose zeros are very easy to
study).
But if, for instance, $\alpha < 0$ and $r_1(x) \approx 0$, then we would be trying to take
the square root of a negative number! 

To get around this, we multiply $g_1(x)$ by $\sigma_1 = \mathrm{sgn}(\alpha)$, which is
$1$ if $\alpha > 0$ and $-1$ if $\alpha < 0$. Then $\sigma_1 g_1(x) = x^2(\sigma_1\alpha + \sigma_1r_1(x))$.
As $r_1(x)$ is the remainder term, we can choose $\delta > 0$ such that $\lvert r_1(x)
\rvert < \lvert \alpha \rvert$ for all $x$ satisfying $\lvert x \rvert < \delta$.
In this interval, $\sigma_1 \alpha + \sigma_1 r_1(x)$ is always positive.
Similarly we multiply $g_2(x)$ by $\sigma_2 = \mathrm{sgn}(\beta)$, so that $\sigma_2
\beta > 0$. As with $r_1$, in a sufficiently small neighborhood around $0$ we have that
$\lvert r_2(y) \rvert < \lvert \beta \rvert$, and in this neighborhood $\sigma_2 \beta +
\sigma_2 r_2(y)$ is always positive.

This allows us to write $u = x\sqrt{\sigma_1(\alpha + r_1(x))}$ and $v = y
\sqrt{\sigma_2(\beta + r_2(y))}$, which is nearly the ideal choices described above.
But to be proper coordinates we require these maps to be invertible.
To study these potential coordinates, we again use the implicit function
theorem~\ref{thm:implicit}.
Namely, we study the zeroes of the two functions
\begin{equation}\label{eq:cov_morse}
  F_1(x, u) = x \sqrt{\sigma_1(\alpha + r_1(x))} - u,
  \qquad
  F_2(y, v) = y \sqrt{\sigma_2(\beta + r_2(y))} - v
\end{equation}
in small neighborhoods of the origin (small enough so that the arguments of the
square roots are always positive).
We calculate that $F_1(0,0) = F_2(0,0) =0$ while
\begin{align*}
  \frac{\partial F_1}{\partial x}(0, 0)  = \sqrt{\sigma_1 \alpha},
  \enskip
  \frac{\partial F_2}{\partial y}(0, 0)  = \sqrt{\sigma_2 \beta},
  \enskip
  \frac{\partial F_1}{\partial u}(0, 0)  = -1,
  \enskip
  \frac{\partial F_2}{\partial v}(0, 0)  = -1,
\end{align*}
all of which are nonzero.
By the implicit function theorem, the first two equalities show that $x = X(u)$ and $y =
Y(v)$ in a neighborhood of $(x, y, u, v) = (0, 0, 0, 0)$. The second two equalities
confirm that $u = U(x)$ and $v = V(y)$ also holds in a neighborhood of the origin.
Further, each of these coordinate maps is continuously differentiable.
Thus $(x, y) \mapsto (u, v)$ is valid change of coordinates.

We have proved that $G(x, y) = g_1(x) - g_2(y) = \sigma_1 u^2 - \sigma_2 v^2$ around a
neighborhood of the origin, i.e.~that the conclusion \eqref{eq:morse} of the Morse lemma
holds.
To continue our investigation of mean value abscissae, we examine
\begin{equation*}
  G(x, y) = \sigma_1 u^2 - \sigma_2 v^2 = 0.
\end{equation*}
There are a few different possibilities depending on the combinations of the signs
$\sigma_1$ and $\sigma_2$.
In terms of the original function $f$, we note that $\sigma_1$ is the sign of $f''(b_0)$
and $\sigma_2$ is the sign of $f'''(c_0)$.
\begin{enumerate}
  \item[(i)]
    If $\sigma_1$ and $\sigma_2$ have opposite signs, then $G(x, y) = 0$ is equivalent to
    $u^2 = - v^2$. The only solution is $(u, v) = (0, 0)$ and this solution is isolated.
  \item[(ii)]
    If $\sigma_1$ and $\sigma_2$ have the same sign, $G(x, y) = 0$ is equivalent to $u^2 =
    v^2$.
    This has two solutions $u = \pm v$, and no other nearby solutions.
\end{enumerate}
Looking again at Figure~\ref{fig:triple}, we see that case (i) corresponds to the right
graph, and case (ii) corresponds to the left graph.

We summarize the results of our exploration in the following theorem.

\begin{theorem}\label{thm:mainII}
  Let $f$ be a four-times continuously differentiable function and fix an interval
  $[a_0,b_0] \subset \mathbb{R}$.
  Suppose $c_0$ is a mean value abscissa for $f$ on the interval $[a_0,b_0]$, and suppose that
  $f''(c_0) = 0$ and $f'(b_0) = f'(c_0)$.
  Finally, suppose that both $f''(b_0)$ and $f'''(c_0)$ are nonzero.
  Then
  \begin{itemize}
    \item If $f''(b_0)$ and $f'''(c_0)$ have opposite signs, then $c_0$ cannot be extended
      to a continuous function $c=C(b)$ near $b_0$.
    \item If $f''(b_0)$ and $f'''(c_0)$ have the same sign, then there are two continuously
      differentiable functions $c=C_1(b)$ and $c=C_2(b)$ solving \eqref{eq:MVT} for $b$
      near $b_0$. There are no other nearby solutions.
  \end{itemize}
\end{theorem}



\section{Analytic Functions.}

Looking back, if $f''(c_0) \ne 0$ we can use Theorem~\ref{thm:mainI}, while if $f''(c_0) =
0$ but $f'''(c_0) \neq 0$ and $f''(b_0) \ne 0$, then we can use Theorem~\ref{thm:mainII}. What if $f''(c_0) =
f'''(c_0) = 0$ but $f^{(4)}(c_0) = f''''(c_0) \ne 0$? Or, even more ambitiously, what if
\begin{align}
  \label{eqn:vanishing}
  f'(c_0) = f''(c_0) = \cdots = f^{(k)}(c_0) = 0
  \quad \text{but} \quad f^{(k+1)}(c_0) \ne 0
\end{align}
for $k=10$ or $k=200$? Ideally we do not want to have to prove a new theorem for each of
these cases. There is also the possibility that \emph{all} of the derivatives of $f$
vanish at $c_0$. For instance this is what happens at $c_0 = 1$ for the classic ``bump
function'', defined to be $\exp(-1/(1-x^2))$ for $-1 < x < 1$ and $0$ otherwise.

We can rule out this last possibility by restricting to analytic functions.
Recall that a function $f$ is \emph{analytic} if the Taylor series for $f$ centered at
each point $x_0$ converges to $f$ in a neighborhood of $x_0$. That is, for each $x_0$, we
have the equality
\begin{align*}
  f(x) = \sum_{n = 0}^\infty \frac{f^{(n)}(x_0)}{n!}(x - x_0)^n
\end{align*}
for all $x$ in a neighborhood of $x_0$. One of the nice things about analytic functions is
that we can only have $f^{(k)}(x_0) = 0$ for all $k \geq 1$ if $f$ is a constant function.
For the rest of this section we will assume that $f$ is a non-constant analytic
function, in which case we can always find a $k$ so that \eqref{eqn:vanishing} holds.

With this assumption in mind, let us return to the function $G(x, y) = F(x + b_0, y + c_0)$,
where $F$ is the implicit function~\eqref{eq:F_def} whose zeros represent solutions to the
mean value theorem relation~\eqref{eq:MVT}. As in the previous section, we will write
$G(x, y) = g_1(x) - g_2(y)$ where $g_1$ and $g_2$ are as in~\eqref{eq:g1g2}.
The given mean value abscissa implies that $g_1(0) = g_2(0) = 0$.
But unlike before, the first nonzero term in the Taylor expansion of $g_2$ is when
$f^{(k+1)}(c_0) \neq 0$, yielding the approximation
\begin{align*}
  g_2(y) \approx \beta_0  y^k.
\end{align*}
where $\beta_0 = f^{(k+1)}(c_0)/k!$ is a nonzero constant.
Similarly picking $\ell$ so that
\begin{align}
  \label{eqn:morevanishing}
  f'(b_0) = f''(b_0) = \cdots = f^{(\ell-1)}(b_0) = 0
  \quad \text{but} \quad f^{(\ell)}(b_0) \ne 0,
\end{align}
a slightly more involved calculation shows that
\begin{align*}
  g_1(x) \approx  \alpha_0 x^\ell,
\end{align*}
where this time the nonzero constant is $\alpha_0 = f^{(\ell)}(b_0)/(\ell!(b-a))$.
We call $\ell$ and $k$ the \emph{order of vanishing} at the origin for $g_1$ and $g_2$, respectively.
Our equation $G(x,y) = 0$ now seems to be approximately
\begin{align*}
  \alpha_0 x^\ell \approx \beta_0 y^k.
\end{align*}
As in our proof of a special case of the Morse lemma, we will make this precise by finding
new coordinates $u,v$ so that $G(x,y)=0$ is \emph{exactly} either $v^k = u^\ell$ or $v^k =
-u^\ell$.

As $f$ is analytic, we can see that both $g_1$ and $g_2$ are analytic. Representing $g_1$
and $g_2$ by their Taylor expansion near $0$, we can write them as
\begin{equation*}
  g_1(x)
  =
  x^\ell \sum_{m=0}^\infty \alpha_m x^m,
  \qquad
  g_2(y)
  =
  y^k \sum_{m=0}^\infty \beta_m  y^m,
\end{equation*}
where $\alpha_0 \ne 0$ and $\beta_0 \ne 0$ were defined above.
As with our special case of the Morse lemma, we now multiply $g_1$ by $\sigma_1 =
\mathrm{sgn} \alpha_0$ and
$g_2$ by $\sigma_2 = \mathrm{sgn} \beta_2$, enabling us to take roots in a neighborhood of $0$.

Taking these roots, we can define two smooth (in fact analytic) functions by
\begin{equation*}
  F_1(x, u) = x \Big( \sigma_1\sum_{m=0}^\infty \alpha_m x^m \Big)^{1/\ell} - u,
  \qquad
  F_2(y, u) = y \Big( \sigma_2\sum_{m=0}^\infty \beta_m y^m \Big)^{1/k} - v,
\end{equation*}
in a neighborhood of the origin. Suppose for the moment that the equations $F_1(x,u)=0$
and $F_2(y,v)=0$ defined a smooth change of coordinates $(x,y) \mapsto (u,v)$. Then, in the
$(u,v)$ coordinates, we would have $g_1 = \sigma_1 u^k$ and $g_2 = \sigma_2 u^\ell$ so
that $G(x,y)=0$ was equivalent to
\begin{align}
  \label{eq:vk}
  u^k = \frac{\sigma_2}{\sigma_1} v^\ell = \pm v^\ell,
\end{align}
analogous to the Morse lemma but with higher powers.

Checking that $F_1(x, u) = 0$ and $F_2(y, v) = 0$ define an invertible change of
coordinates $x = X(u)$ and $y = Y(v)$ can be proved from the implicit function
theorem Theorem~\ref{thm:implicit}, and in this case the proof is almost identical to the
proof for~\eqref{eq:cov_morse}, the change of coordinates from our consideration of
Morse's lemma.
Thus $(x,y) \mapsto (u, v)$ is a valid change of coordinates, and we can study $G(x, y)
= 0$ by studying solutions to~\eqref{eq:vk}.

Thinking about the graph of \eqref{eq:vk} for different values of $k$ and $\ell$, and
different combinations of signs of $\sigma_1$ and $\sigma_2$, we find that
\begin{enumerate}
\item[(i)] If $k$ is odd, then there is one continuous solution $v=V_1(u)$ of
  \eqref{eq:vk} in a neighborhood of the origin, and no other nearby solutions.
\item[(ii)] If $k$ and $\ell$ are both even and $\sigma_2/\sigma_1 = +1$ then there are
  two continuous solutions $v=V_1(u)$ and $v=V_2(u)$ of \eqref{eq:vk} in a neighborhood of
  the origin, and no other nearby solutions.
\item[(iii)] If $k$ and $\ell$ are both even and $\sigma_2/\sigma_1 = -1$, then the origin
  is an isolated solution of \eqref{eq:vk}.
\item[(iv)] If $k$ is even and $\ell$ is odd, then there are two continuous solutions
  $v=V_1(u)$ and $v=V_2(u)$ of \eqref{eq:vk}, but they are only defined in a one-sided
  neighborhood of the origin where $(\sigma_2/\sigma_1) u \ge 0$.
\end{enumerate}
Since we can always write $y=Y(v)$ and $u=U(x)$, continuously solving for $v$ in terms of
$u$ is equivalent to continuously solving for $y$ in terms of $x$. As
we're solving $G(x,y) = F(x+b_0,y+c_0) = 0$, this is equivalent to continuously
solving~\eqref{eq:MVT} for the abscissa $c=c_0+y$ in terms of the endpoint $b=b_0+x$ for
$x$ in a neighborhood of $(b,c)= (b_0,c_0)$.

Thus we can find a continuous choice of mean value abscissa near any point where
(i) or (ii) hold. The following lemma tells us that there
there is \emph{always} at least one such point (i) holds.

\begin{lemma}
  Let $f$ be a non-constant analytic function satisfying $f(a_0) = f(b_0) = 0$.
  Then there is a mean value abscissa $c_0$ of $f$ on $[a_0, b_0]$ such that $f'(c_0) = 0$
  and such that the smallest $k \geq 1$ such that $f^{(k)}(c_0) \neq 0$ is even,
  i.e.\ such that the order of vanishing of $f'$ at $c_0$ is odd.
\end{lemma}

\begin{proof}
  Since $f$ is non-constant, $f$ takes an absolute maximum or an absolute minimum at a
  point $c$ within the interval $(a_0, b_0)$. Notice that $c$ is not an endpoint of
  the interval since $f(a_0) = f(b_0) = 0$ and the function is not constant.
  We choose $c_0$ to be this point $c$.

  As $c_0$ is an extremum, $f'(c_0) = 0$.
  From its Taylor expansion, we see that near $c_0$, $f$ is very closely approximated by
  $f(c_0) + a_k(x - c)^{k}$, where $a_k = f^{(k)}(c_0) / (k!)$. If $k$ were odd, then $f$
  would be strictly increasing or decreasing at $c_0$, contradicting the fact that it has
  a local extremum there.
\end{proof}

With this lemma, we are now ready to complete our study of when there exist continuous
choices of mean value abscissae $c=C(b)$.
Recall that we assume without loss of generality that $f(a_0) = f(b_0) = 0$.
Let $c_0$ be a mean value abscissa for $f$ on $[a_0, b_0]$ such that the order of
vanishing of $f'$ at $c_0$ is odd, as guaranteed by the lemma.
The order of vanishing of $f'$ at $c_0$ is the same as the order of vanishing of $g_2$ at
$0$ by construction, and thus the lemma indicates that the $k$ appearing in~\eqref{eq:vk}
is odd. Thus we are in case (i), and we can uniquely solve \eqref{eq:MVT} for $c=C(b)$ in
a neighborhood of $b_0$.
This completes the proof of the following theorem.

\begin{theorem}\label{thm:analytic_main}
  Let $f$ be real analytic on the interval $[a_0,b_0]$.
  Then there exists at least one mean value abscissa $c_0 \in (a_0,b_0)$ such that $c_0$
  is a mean value abscissa for $f$ on $[a_0,b_0]$, and for which there exists a
  continuous function $c = C(b)$ such that
  \begin{equation*}
    \frac{f(b) - f(a)}{b - a} = f'(C(b))
  \end{equation*}
  for all $b$ in a neighborhood of $b_0$. There are no other solutions near $(b_0,c_0)$.
\end{theorem}

\begin{figure}[h]
  \includegraphics[scale=0.9]{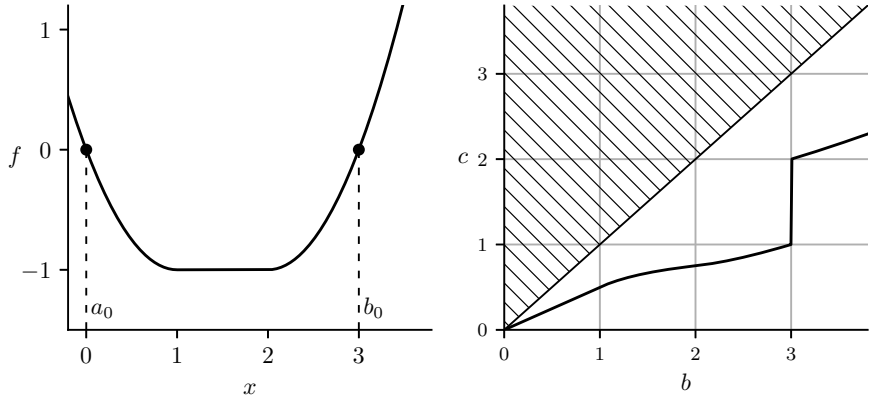}
  \centering
  \caption{%
  A smooth function where no continuous choice of mean value abscissa exists.
  \label{fig:smooth}}
\end{figure}

\begin{remark}

As a final note, we note that being ``merely'' infinitely differentiable is not strong
enough to guarantee that there is always a choice of a mean value abscissa with a
continuous dependence on the right endpoint.
For a counterexample, see Figure~\ref{fig:smooth}.
One can construct a smooth function of this shape from bump functions.
On the indicated interval $[0, 3]$, every value $c_0$ with $1 \leq c_0 \leq 2$ is a valid mean
value abscissa; this is reflected in the mean value abscissa plot on the right by a
vertical line segment from $(3, 1)$ to $(3, 2)$.
For $b$ just to the left of $b_0$, the slope of the secant line is negative, and for $b$
just to the right of $b_0$, the slope of the secant line is positive.
But any value $c_0$ is either at least distance $1/2$ away from a point $c$ where $f'(c) <
0$ or a point $c$ where $f'(c) > 0$.
There is no continuous choice of mean value abscissa for this function.

\end{remark}

\section{Reflection and Further Questions.}

The major selling-point of Theorem~\ref{thm:analytic_main} is that it gives us a
continuous choice $c=C(b)$ of mean value abscissa without any assumptions on $f$ other
than analyticity.
On the other hand, in cases where we \emph{do} know more about our favorite abscissa
$c_0$, the classification (i)--(iv) of the curves~\eqref{eq:vk} gives much more
information about nearby solutions.
And for the local picture, we should expect ``most'' points for ``most'' intervals to not
be degenerate enough that Theorem~\ref{thm:mainI} and Theorem~\ref{thm:mainII} both fail.

There are many more questions that one could ask about the set of solutions
to~\eqref{eq:MVT}.
Firstly, what if you allow the left endpoint $a$ to vary as well as $b$ and look for
continuous choices $c=C(a,b)$?
The techniques we have used will still be very powerful, but for instance the
decomposition $G(x,y) = g_1(x) - g_2(y)$ in the above section will no longer be as simple.

One could also study the \emph{global} structure of the solution sets shown in
Figures~\ref{fig:curlym} and~\ref{fig:triple}. How many different connected components are
there?  In what ways can they ``begin'' and ``end''? Answering such questions will require 
very different techniques.

\section*{Acknowledgements}

D Lowry-Duda was supported by the  National Science Foundation Graduate Research
Fellowship Program under Grant No.\ DGE 0228243 and the EPSRC Programme Grant
EP/K034383/1 LMF: L-Functions and Modular Forms. 

Miles H.\ Wheeler was supported by the National Science Foundation under Grant No.\ DMS-1400926.

We also thank the many contributors to the python programming packages NumPy, SymPy, and
matplotlib, as we used this software for our own exploration and to create the functions
and figures in this article. A copy and description of the code used for this article are
available at\\
\url{http://davidlowryduda.com/choosing-functions-for-mvt-abscissa/}.


\end{document}